\documentclass[reqno,12pt]{amsart}
\usepackage{amsmath, latexsym, amsfonts, amssymb, amsthm, amscd}
\usepackage{multirow}

\usepackage{xcolor}

\newcommand{\change}{}
\usepackage[utf8]{inputenc}

\usepackage{hyperref}
\usepackage{graphics,epsf,psfrag}
\numberwithin{equation}{section}

\newtheorem{thm}{Theorem}[section]
\newtheorem{lem}[thm]{Lemma}
\newtheorem{prop}[thm]{Proposition}
\newtheorem{cor}[thm]{Corollary}
\newtheorem{rem}[thm]{Remark}

\newcommand{\ind}{\mathbf{1}}

\newcommand{\R}{\mathbb{R}}
\newcommand{\Z}{\mathbb{Z}}
\newcommand{\N}{\mathbb{N}}
\renewcommand{\tilde}{\widetilde}

\newcommand{\cF}{{\ensuremath{\mathcal F}} }
\newcommand{\cP}{{\ensuremath{\mathcal P}} }

\newcommand{\cC}{{\ensuremath{\mathcal C}} }

\newcommand{\cL}{{\ensuremath{\mathcal L}} }

\newcommand{\cW}{{\ensuremath{\mathcal W}} }
\newcommand{\cM}{{\ensuremath{\mathcal M}} }

\newcommand{\bP}{{\ensuremath{\mathbf P}} }
\newcommand{\bE}{{\ensuremath{\mathbf E}} }

\DeclareMathSymbol{\leqslant}{\mathalpha}{AMSa}{"36} 
\DeclareMathSymbol{\geqslant}{\mathalpha}{AMSa}{"3E} 
\DeclareMathSymbol{\eset}{\mathalpha}{AMSb}{"3F}     
\renewcommand{\leq}{\;\leqslant\;}                   
\renewcommand{\geq}{\;\geqslant\;}                   
\newcommand{\dd}{\text{\rm d}}             

\newcommand{\bbE}{{\ensuremath{\mathbb E}} }

\newcommand{\bbP}{{\ensuremath{\mathbb P}} }

\newcommand{\bbR}{{\ensuremath{\mathbb R}} }

\newcommand{\bbT}{{\ensuremath{\mathbb T}} }

\newcommand{\norm}[1]{\left\lVert#1\right\rVert}
\newcommand{\cutnorm}[2]{d_{\square}\left(#1,#2\right)}

\makeatletter
\def\captionfont@{\footnotesize}
\def\captionheadfont@{\scshape}

\long\def\@makecaption#1#2{%
  \vspace{2mm}
  \setbox\@tempboxa\vbox{\color@setgroup
    \advance\hsize-6pc\noindent
    \captionfont@\captionheadfont@#1\@xp\@ifnotempty\@xp
        {\@cdr#2\@nil}{.\captionfont@\upshape\enspace#2}%
    \unskip\kern-6pc\par
    \global\setbox\@ne\lastbox\color@endgroup}%
  \ifhbox\@ne 
    \setbox\@ne\hbox{\unhbox\@ne\unskip\unskip\unpenalty\unkern}%
  \fi
  \ifdim\wd\@tempboxa=\z@ 
    \setbox\@ne\hbox to\columnwidth{\hss\kern-6pc\box\@ne\hss}%
  \else 
    \setbox\@ne\vbox{\unvbox\@tempboxa\parskip\z@skip
        \noindent\unhbox\@ne\advance\hsize-6pc\par}%
\fi
  \ifnum\@tempcnta<64 
    \addvspace\abovecaptionskip
    \moveright 3pc\box\@ne
  \else 
    \moveright 3pc\box\@ne
    \nobreak
    \vskip\belowcaptionskip
  \fi
\relax
}
\makeatother
\def\writefig#1 #2 #3 {\rlap{\kern #1 truecm
\raise #2 truecm \hbox{#3}}}

\title[Interacting Oscillators on Dense Random Graphs]{Weakly Interacting Oscillators \\ on Dense Random Graphs}

\author{Gianmarco Bet}
\address{
  Dipartimento di Matematica e Informatica ``Ulisse Dini", Università degli Studi di Firenze, Firenze, Italy
}
\author{Fabio Coppini}
\address{
	Dipartimento di Matematica e Informatica ``Ulisse Dini", Università degli Studi di Firenze, Firenze, Italy
}
\author{Francesca R.~Nardi}
\address{
  Dipartimento di Matematica e Informatica ``Ulisse Dini", Università degli Studi di Firenze, Firenze, Italy, and
  Department of Mathematics and Computer Science, Eindhoven University of Technology, Eindhoven, the Netherlands
}

\begin{document}

\begin{abstract}
	We consider a class of weakly interacting particle systems of mean-field type. The interactions between the particles are encoded in a graph sequence, i.e., two particles are interacting if and only if they are connected in the underlying graph. We establish a Law of Large Numbers for the empirical measure of the system that holds whenever the graph sequence is convergent to a graphon. The limit is the solution of a non-linear Fokker-Planck equation weighted by the (possibly random) graphon limit. In contrast with the existing literature, our analysis focuses on both deterministic and random graphons: no regularity assumptions are made on the graph limit and we are able to include general graph sequences such as exchangeable random graphs. Finally, we identify the sequences of graphs, both random and deterministic, for which the associated empirical measure converges to the classical McKean-Vlasov mean-field limit.
\end{abstract}

\dedicatory{To our friend and colleague Carlo Casolo}

\maketitle

\section{Introduction, Organization and Set-up}

In the last twenty years there has been a growing interest in complex networks and inhomogeneous particle systems. The classical mean-field framework (see, e.g., \cite{cf:Oel84,cf:szni}) in which all the particles are connected with each other, has been extended to include interactions described by general networks. In these more general models, the interaction between two particles depends on the weight of the edge connecting the two in an underlying network. {\change Applications of these models include mean-field games \cite{cf:CarmonaI, cf:CH18},  synchronization phenomena \cite{cf:Arenas}, neuroscience \cite{cf:RTPK16}, and statistical mechanics \cite{van_enter_gibbsianness_2009} among others.}

The first mathematically rigorous results {\change on McKean-Vlasov particle systems and graphs appeared only recently \cite{cf:BBW,cf:DGL}: they consider certain graph sequences with diverging average degree. Under suitable homogeneity conditions on the degrees, the system is described by the classical mean-field limit \cite{cf:szni} as the number of particles tends to infinity.} The cited works leave several relevant questions unanswered: is it possible to characterize the graph sequences for which the system converges to the mean-field limit? How sensitive is the system dynamics to the degree inhomogeneity in the underlying graph? {\change The existing literature, see, e.g., \cite{cf:BCW20, cf:CDG,cf:L18}, focuses on a large class of particle systems and different graph regimes, i.e., from dense to almost sparse graphs, yet they are unable to fully include the well-known class of exchangeable random graphs \cite{cf:DJ08} in their result.
  
Our aim is to prove that a Law of Large Numbers for weakly interacting particle systems can be proven for any dense random graph sequence, including exchangeable random graphs. To keep the focus on the random graph limit viewpoint, we restrict the class of interacting particles to weakly interacting oscillators, a class of systems that shows a remarkably rich behavior \cite{cf:Arenas, cf:RTPK16}. We note that, at the cost of more technicalities, our result could be extended to more general particle systems.}
The main result of this work is a Law of Large Numbers for the empirical measure under the only assumption that the graph sequence converges in probability in the space of graphons, see \cite{cf:Lov} for a complete tour on graphons. Namely, we tackle the challenging case of a random graph limit, which includes pseudo-random graphs (see, e.g., \cite{cf:BBCN16,cf:CGW89}) and exchangeable random graphs (see, e.g., \cite{cf:DJ08}). To the authors' knowledge, this is the first result in the literature that explicitly links {\change random graph limits to empirical measures}.

As a byproduct, we are able to identify deterministic and random graph sequences for which the particle system behavior is approximately mean-field. We also provide an example of a particle system undergoing a phase transition, where the critical threshold depends on the randomness of the graph limit, see Proposition \ref{p:GPA} and the example below.

\subsection{Unlabeled graphons and empirical measures}

To better illustrate our contribution, a clarification on the term graphon is in order. {\change In the literature of McKean-Vlasov particle systems, the term \emph{graphon} is usually employed for what the classical graph limits theory (we refer to \cite{cf:Lov, cf:DJ08}) indicates as \emph{labeled graphon}, i.e.,} a symmetric, bounded and measurable function $W: [0,1] \times [0,1] \to \R$. The space $[0,1]$ is known as the label space and it is used as particle index: roughly speaking, every $x \in [0,1]$ corresponds to a particle and the interaction strength between two particles $x$ and $y$ is given by the value of $W(x,y)$.

The mathematical space of graphons is not merely given by labeled graphons but constructed via certain equivalence classes, i.e.,
\begin{equation*}
    \tilde{W} = \{W^\varphi \text{ with } \varphi : [0,1] \to [0,1] \text{ measure-preserving map}\},
\end{equation*}
where $W^\varphi (x, y) = W (\varphi(x), \varphi(y))$ for every $x$ and $y$ in $[0,1]$. The measure-preserving map $\varphi$ corresponds to a label permutation in a finite graph {\change (note that $\varphi$ need not be invertible in the infinite setting, see \cite[Proposition 7.10]{cf:Lov})}. In other words, $\tilde{W}$ contains the connection information of some labeled graphon $W$, while being independent of its labeling. Following \cite{cf:DJ08}, we refer to this object as  an \emph{unlabeled graphon}. The space of graphons is the space of unlabeled graphons with a suitable distance, see Subsection \ref{ss:graphons} for further details.

The notion of labeling is relevant when studying interacting particle systems. On the one hand, one is interested in understanding the correlation between two particles and needs their labels to be fixed to suitably describe it. On the other hand, one wants to describe the macroscopic behavior, e.g., by studying the empirical measure, and there the labels do not play any role. Hence, depending on the question one is interested in, the suitable modelling choice might be labeled or unlabeled graphs. From a graphon viewpoint, considering unlabeled graphons is not a technical recourse but, in fact, represents the stepping stone of the theory and the key ingredient to many of its interesting properties. As it will become clear later, when describing the dynamics of a particle population, the concept of labeled graphon is not satisfactory: different labeled graphons can lead to exactly the same dynamics!

Our contribution stems from the previous considerations. We refer to Section \ref{ss:literature} for a comparison with the existing literature.

\subsection{Organization}
We now present the set-up and notation used, as well as the various distances between probability measures that will considered in the sequel. We give a short introduction to graphons and the main tools needed for this work {\change in Section \ref{sec:setting}}.

In Section \ref{sec:model_main_results} we define the interacting particle system and the associated non-linear process. Existence, uniqueness and stability results for the non-linear process are presented right after, see Theorem \ref{thm:nonLin2}. Our main result, Theorem \ref{thm:conv}, is given in Subsection \ref{ss:convergence}. Exchangeable random graphs are then discussed together with a propagation of chaos result, see Subsection \ref{ss:exchangeable}. The following Subsection \ref{ss:examples} is devoted to the comparison with the classical mean-field behavior and to a few important consequences of Theorem \ref{thm:conv}; the discussion is supported by two explanatory examples.

In Section \ref{s:nonLin} we focus on the non-linear process. In particular, we discuss its relationship with other characterizations already known in the literature. The proofs of the main results for the non-linear process are given in Subsection \ref{pss:nonLin2}.

Section \ref{pss:empMes} contains the proof of Theorem \ref{thm:conv}. Finally, in Appendix \ref{annex:graph} we give a useful characterization of convergence in probability for random graph sequences.

\section{Setting and notations}\label{sec:setting}
We consider particle dynamics occurring on a finite time interval, say $[0,T]$, which we fix once and for all. We work on the filtered probability space $(\Omega, \cF, \{\cF_t\}_{t \in [0,T]}, P)$, where $\{\cF_\cdot\}$ is a filtration satisfying the usual conditions.

We use two different notations for expressing conditional probabilities: the one referring to Brownian motions and initial conditions is denoted by $\bP$, its expectation by $\bE$; the one referring to the randomness in the graph sequences, and/or in its limit object, is denoted by $\bbP$, its expectation by $\bbE$.

The interval $I:=[0,1]$ represents the space of (continuous) labels. {\change We study functions with values in the one-dimensional torus denoted by $\bbT:= \R/(2\pi\Z)$. The space of continuous functions from $[0,T]$ to $\bbT$ is denoted by $\cC([0,T],\bbT)$ and it is endowed with the supremum norm. For $\alpha>0$, the space of $\alpha$-H\"{o}lder continuous function from $\bbT^2$ to $\R$ is denoted by $\cC^{\alpha}(\bbT^2)$. Given a separable metric space $E$, the space of probability measures over $E$ is denoted by $\cP(E)$.}

The various constants throughout the paper are always denoted by $C$ or $C'$ and may vary from line to line. An explicit dependence on a parameter $\alpha$ will be denoted by $C_\alpha$.

\subsection{Distance between probability measures}

For two probability measures $\bar{\mu},\, \bar{\nu} \in \cP(\cC([0,T], \bbT))$, we define their distance by
\begin{equation}
\label{d:wass_T}
D_T (\bar{\mu},\bar{\nu}) := \inf_{m \in \gamma(\bar{\mu},\bar{\nu})} \left\{ \int \sup_{t \in [0,T]} \left|x_t -y_t\right|^2 m (\dd x, \dd y) \right\}^{1/2},
\end{equation}
where $\gamma(\bar{\mu},\bar{\nu})$ is the space of probability measures on $\cC([0,T], \bbT) \times \cC([0,T], \bbT)$ with first marginal equal to $\bar{\mu}$ and second marginal equal to $\bar{\nu}$. This definition coincides with the 2-Wasserstein distance between probability measures. The right-hand side of \eqref{d:wass_T} can be rewritten as
\begin{equation}
\label{d:wass_T2}
D_T (\bar{\mu},\bar{\nu})  = \inf_{X,Y} \left\{ \bE \left[ \sup_{t \in [0,T]} \left|X_t - Y_t\right|^2 \right] : \cL(X) = \bar{\mu}, \, \cL(Y) = \bar{\nu} \right\}^{1/2}
\end{equation}
where the infimum is taken on all random variables $X$ and $Y$ with values in $\cC([0,T], \bbT)$ and law $\cL$ equal to $\bar{\mu}$ and $\bar{\nu}$ respectively. From \eqref{d:wass_T} we obtain that for every $s \in[0,T]$
\begin{equation}
\label{d:lips_dist}
\sup_{f} \left|\int_\bbT f(\theta) \, \bar{\mu}_s (\dd \theta) - \int_\bbT f(\theta) \, \bar{\nu}_s (\dd \theta)\right| \leq D_s(\bar{\mu},\bar{\nu}),
\end{equation}
where the supremum is taken over all {\change 1-Lipschitz functions} from $\bbT$ to $\R$. Observe that these definitions make sense also with $T=0$ and $\cC([0,T], \bbT)$ replaced by $\bbT$.

\subsection{Distance between finite graphs}

We denote $[n]:=\{1,\dots,n\}$ for $n \in \N$. Let $\xi$ be a graph on $n$ vertices. With an abuse of notation, we let $\xi$ denote its adjacency matrix as well, i.e., $\xi = \{\xi_{ij}\}_{i,j \in [n]}$. We consider simple undirected graphs so that $\xi_{ij} =\xi_{ji}$  and $\xi_{ii} = 0$ for all $1 \leq i \leq j\leq n$.


Let $A=\{A_{ij}\}_{i,j\in[n]}$ be a $n\times n$ real matrix. The \textit{cut-norm} of $A$ {\change is a well-known norm for (not necessarily square) finite matrices (see, e.g., \cite{cf:AN06, cf:Lov})}, it is defined as
\begin{equation}
    \label{d:cutnorm}
    \norm{A}_{\square} := \frac 1{n^2} \max_{S,T \subset [n]} \left|  \sum_{i \in S, j\in T} A_{ij} \right|.
\end{equation}
It is well-known that this norm is equivalent to the $\ell_\infty\to \ell_1$ norm \cite{cf:AN06}
\begin{equation}
    \norm{A}_{\infty \to 1} := \sup_{s_i,t_j \in \{\pm 1\}} \sum_{i,j=1}^n A_{ij} s_i t_j.
\end{equation}
For two graphs $\xi$ and $\xi'$ on the same set of vertices, we define the distance $d_\square$ as
\begin{equation}
    \cutnorm{\xi}{\xi'} := \norm{\xi- \xi'}_{\square}.
\end{equation}

\subsection{Labeled and unlabeled graphons}
\label{ss:graphons}

{\change The following definitions come from the book of Lovász \cite{cf:Lov}, whose notation we adopt}.  Recall that $I=[0,1]$ and let 
\[
\cW := \{W:I^2 \to \R \text{ bounded symmetric and measurable}\}
\]
be the space of kernels, we tacitly consider two kernels to be equal if and only if the subset of $I^2$ where they differ has {\change zero Lebesgue measure on $I^2$}. A graphon is a kernel $W$ such that $0\leq W \leq 1$. Let $\cW_0$ denote the space of graphons. The cut-norm of $W\in\cW$ is defined as
\begin{equation}
    \label{d:cutnormW}
    \norm{W}_{\square} := \max_{S,T \subset I} \left| \int_{S\times T} W(x,y) \dd x \dd y \right|
\end{equation}
where the maximum is taken over all measurable subsets $S$ and $T$ of $I$. It is well known that $\norm{W}_\square$ is equivalent to the norm of $W$ seen as an operator from $L^\infty(I) \to L^1 (I)$  \cite[Theorem 8.11]{cf:Lov}. This is defined as
\begin{equation}
    \label{d:opNorm}
    \norm{W}_{\infty \to 1} := \sup_{\norm{g}_\infty \leq 1}  \norm{Wg}_1,
\end{equation}
where $(Wg)(x) :=\int_I W(x,y)g(y) \dd y$ for $x \in I$ and $g \in L^\infty (I)$.

The metric induced by $\norm{\cdot}_\square$, or equivalently by $\norm{\cdot}_{\infty \to 1}$, in the space of graphons $\cW_0$ is again denoted by $\cutnorm{\cdot}{\cdot}$. Definitions \eqref{d:cutnorm} and \eqref{d:cutnormW} are consistent in the sense that to each graph $\xi$ is associated a graphon $W_\xi \in \cW_0$ such that $\norm{\xi}_\square = \norm{W_\xi}_\square$. The graphon $W_\xi$ is usually defined a.e.~as
\begin{equation}
    \label{d:graphGraphon}
    W_\xi (x,y) = \sum_{i,j=1}^n \xi_{ij} \, \ind_{ [\frac{i-1}n , \frac in)\times [\frac {j-1}n \frac jn)} (x,y), \quad \text{ for } x,y \in I.
\end{equation}
%

Note that $W_\xi$ depends on the labeling of $\xi$. Indeed, different labelings of $\xi$ yield graphs which have large $d_\square$-distance in general. This motivates the definition of the so-called \textit{cut-distance}. For two graphs $\xi, \, \xi'$ with the same number of nodes, the cut-distance is defined as
\begin{equation}
    \label{d:cut-distanceGraph}
    \hat{\delta}_\square (\xi, \xi') := \, \min_{\hat{\xi}'} \, d_\square (\xi, \hat{\xi}'),
\end{equation}
where the minimum ranges over all labelings of $\xi'$. The cut-distance is also defined for graphons as follows. For two graphons $W,V \in \cW_0$, their cut-distance is
\begin{equation}
    \label{d:cut-distance}
    \delta_\square (W,V) := \, \min_{\varphi \in S_I} \, \cutnorm{W}{V^\varphi},
\end{equation}
where the minimum ranges over $S_I$ the space of invertible measure preserving maps from $I$ into itself and where $V^\varphi (x,y) := V(\varphi(x),\varphi(y))$ for $x,y \in I$.

\begin{rem}
    \label{rem:cut-distance}
    There are at least two ways to compare the graphs $\xi, \xi'$ as unlabeled objects: either by directly computing their distance $\hat{\delta}_\square$ or by computing the distance $\delta_\square$ between $W_\xi$ and $W_{\xi'}$. These turn out to be equivalent as the number of vertices tends to infinity \cite[Theorem 9.29]{cf:Lov}. Formally, for every two graphs $\xi, \xi'$ on $n$ vertices, it holds that
    \begin{equation}
        \delta_\square (W_\xi, W_{\xi'}) \leq \hat{\delta}_\square (\xi, \xi') \leq \delta_{\square} (W_\xi, W_{\xi'}) + \frac {17}{\sqrt{\log n}}.
    \end{equation}
    We always write $\delta_\square (\xi, \xi') :=\delta_{\square} (W_\xi, W_{\xi'})$.
\end{rem}

Contrary to $d_\square$, the cut-distance $\delta_\square$ is a pseudometric on $\cW_0$ since the distance between two different graphons can be zero. This leads to the definition of the \textit{unlabeled} graphon $\tilde{W}$ associated to $W$. For a graphon $W$, $\tilde{W}$ is defined as the equivalence class of $W$ including all $V\in \cW_0$ such that $\delta_\square(W,V)=0$. For notation's sake, we  drop both the superscript and the adjective unlabeled when the context is clear.

The quotient space obtained in such a way is denoted by $\tilde{\cW}_0$ and we refer to it as the space of unlabeled graphons. A celebrated result of graph limits theory is that $(\tilde{\cW}_0, \delta_\square)$ is a compact metric space \cite[Theorem 9.23]{cf:Lov}.

We are not going into the details of graph convergence for which we refer to the exhaustive reference \cite{cf:Lov}. We only recall that a sequence of graphs $\{\xi^{(n)}\}_{n \in \N}$ converges to the graphon $W \in \tilde{\cW}_0$ if and only if $\delta_\square(W_{\xi^{(n)}},W)\to 0$ as $n\to\infty$ \cite[Theorem 11.22]{cf:Lov}.


\section{The Models and Main Results}\label{sec:model_main_results}

\subsection{The models}
We introduce our two main models: a weakly interacting particle system \eqref{eq:g} and a non-linear process \eqref{eq:nonLin2}.


\medskip

\subsubsection{Weakly interacting oscillators on graphs}
Let $\{\xi^{(n)}\}_{n \in \N}$ be a sequence of undirected, labeled graphs. For $n \in \N$, the adjacency matrix of $\xi^{(n)}$ is given by the $n\times n$ symmetric matrix $\{\xi^{(n)}_{ij}\}_{i,j =1, \dots, n}$ where $\xi^{(n)}_{ij}=1$ whenever the vertices $i$ and $j$ are connected and $\xi^{(n)}_{ij}=0$ otherwise.
Let $\{\theta^{i,n} \}_{i=1,\dots,n}$ be the family of oscillators on $\bbT^n$ that satisfy
\begin{equation}
\label{eq:g}
\begin{cases}
\dd{\theta}_{t}^{i,n} =\, F(\theta^{i,n}_t) \dd t +  \frac{1}{n} \sum_{j=1}^{n} \, \xi^{(n)}_{ij} \, \Gamma (\theta_{t}^{i,n}, \theta_{t}^{j,n}) \dd t + \dd B_{t}^{i}, \quad &0<t < T,\\
\; \; \theta_0^{i,n} =\, \theta_0^i,   &i \in  \{1,\dots,n\},
\end{cases}
\end{equation}
where $F$ and $\Gamma$ are bounded, uniformly Lipschitz functions and $\left\{B^i\right\}_{i\in\N}$ a sequence of independent and identically distributed (IID) Brownian motions on $\bbT$. The initial conditions $\left\{\theta_0^i \right\}_{i\in\N}$ are IID random variables sampled from some probability distribution $\bar{\mu}_0 \in \cP(\bbT)$ which is fixed once {\change and} for all.

\begin{rem}
  {\change
  At the cost of more technical details, one can relax the hypothesis on the initial conditions and require independent but not identically distributed $\{\theta^i_0\}_{i \in \N}$. We rather keep the proofs simple, following the classical propagation of chaos arguments, see, e.g., \cite{cf:szni}.
}
\end{rem}
  
Many interesting examples of interacting oscillators fit this framework such as the Kuramoto model, the plane rotator model and other generalizations, see, e.g., \cite[\S1.2]{cf:DGL}, \cite{cf:BBCN16} and also Subsection \ref{ss:examples}. {\change The applications of interacting oscillators range from modelling synchronisation phenomena, see, e.g., \cite{cf:Arenas}, to  neuroscience, e.g., \cite{cf:RTPK16}, as well as for studying statistical mechanics properties of complex systems \cite{van_enter_gibbsianness_2009}.}

We are interested in studying the empirical measure associated to \eqref{eq:g}. This is defined as the (random) probability measure on $\bbT$ such that
\begin{equation}
\label{d:emp}
\mu^n_t : = \frac 1n \sum_{j=1}^n \delta_{\theta^{j,n}_t},
\end{equation}
for every $t\in[0,T]$. {\change Alternatively, one can also consider the empirical measure $\mu^n$ of the trajectories, i.e., on $\mu^n = \frac 1n \sum_{j=1}^n \delta_{\theta^{j,n}}\in \cP(\cC([0,T], \bbT))$, where $\theta^{j,n} = (\theta^{j,n}_t)_{t\in[0,T]} \in \cC([0,T], \bbT)$ is the trajectory on $[0,T]$ of particle $j$ for $j=1, \dots, n$. The two are equally considered in the literature, see, e.g., \cite[Remark 1.1]{cf:CDG}, for the relation between the two.}

\subsubsection{The non-linear process}

Fix a graphon $W \in \tilde{\cW}_0$ and a uniform random variable $U$ on $I$. Consider the solution $\theta = \{\theta_t\}_{t \in [0,T]}$ to the following system
\begin{equation}
\label{eq:nonLin2}
\begin{cases}
    \, \theta_t = \theta_0 + \int_0^t F(\theta_s) \dd s + \int_0^t \int_I W(U,y) \int_\bbT \Gamma(\theta_s, \theta) \mu^y_s (\dd \theta) \dd y \, \dd s + B_t, \\
    \mu^y_t = \cL (\theta_t | U=y), \quad \text{ for }y \in I, \, t \in [0,T],
\end{cases}
\end{equation}
where {\change the initial condition $\theta_0$ has law $\cL (\theta_0)=\bar{\mu}_0$ and is taken to be independent of $U$. The Brownian motion $B$ is independent of the previous sequence $\{B^i\}_{i \in \N}$, of $U$, of the initial condition $\theta_0$ and also independent of the sequence $\xi^{(n)}$.}

The next theorem on existence and uniqueness of the solution to equation \eqref{eq:nonLin2} is proven in Section \ref{s:nonLin}, together with well-posedness with respect to $W$. For the latter, see Remark \ref{rem:varphi}.

\begin{thm}
    \label{thm:nonLin2}
    Suppose that $F$ and $\Gamma$ are bounded, uniformly Lipschitz functions. For every uniform random variable $U$ on $I$, there exists a unique {\change pathwise} solution to \eqref{eq:nonLin2}.
    
    If $\bar{\mu} \in \cC([0,T], \cP(\bbT))$ denotes the law of the solution $\theta = \{\theta_t\}_{t \in [0,T]}$ and $\mu^x$ the law of $\theta$ conditioned on $U=x$, then $\bar{\mu}$ solves the following non-linear Fokker-Planck equation in the weak sense
    \begin{equation}
        \label{eq:pde_mean}
        \partial_t \bar{\mu}_t (\theta) = \frac 12 \partial^2_\theta \bar{\mu}_t (\theta) - \partial_\theta \left[\bar{\mu}_t(\theta) F(\theta) \right] - \partial_\theta \left[ \int_{I\times I} W(x,y) \, \mu^x_t (\theta) \int_\bbT \Gamma(\theta, \tilde{\theta}) \, \mu^y_t(\dd \tilde{\theta}) \dd y \, \dd x \right],
    \end{equation}
    with initial condition $\bar{\mu}_0 \in \cP(\bbT)$.
\end{thm}

\subsection{Main result}
\label{ss:convergence}

We are now able to present our main result. Afterwards, we present an application to exchangeable random graphs and a propagation of chaos result.

\begin{thm}
    \label{thm:conv}
    Assume the hypothesis of Theorem \ref{thm:nonLin2} and further suppose that $\Gamma \in \cC^{1+\varepsilon}(\bbT^2)$ for some $\varepsilon>0$. Let $\{\xi^{(n)}\}_{n \in \N}$ be a sequence of random graphs. Assume that there exists a random variable $W$ in $\tilde{\cW}_0$ to which $\xi^{(n)}$ converges in $\bbP$-probability, or equivalently such that
    \begin{equation}
    \label{d:graphConv}
    \lim_{n \to \infty} \bbE \left[ \delta_\square \left(\xi^{(n)}, W \right)\right] = 0.
    \end{equation}
    Suppose that the initial conditions $\{\theta^i_0\}_{i \in \N}$ are IID random variables with law $\bar{\mu}_0$ and that are independent of $\{\xi^{(n)}\}_{n \in \N}$. Then,
    \begin{equation}
    \label{eq:convergence}
    \mu^n \longrightarrow \bar{\mu}, \quad \text{in $\bP \times \bbP$-probability, as $n \to \infty$},
    \end{equation}
    where the convergence is in $\cP(\cC([0,T], \bbT)))$ and $\bar{\mu}$ is a random variable depending only on the randomness of $W$, i.e., for almost every $\omega \in \Omega$, $\bar{\mu} (\omega)$ solves equation \eqref{eq:pde_mean} starting from $\bar{\mu}_0$, with graphon $W(\omega)$.
\end{thm}

The hypothesis \eqref{d:graphConv} on the graph convergence is very general. Indeed, it holds for any $\xi^{(n)}$ in $\cW$, meaning that $\xi^{(n)}$ can take values in $\R$ rather than $\{0,1\}$, recall {\change the definition of kernels in} subsection \ref{ss:graphons}. Moreover, the sequence $\xi^{(n)}$ can be deterministic or random; in case it is random, we also cover the convergence in probability to a random graphon limit $W$. In this last case, well-posedness and measurability of equation \eqref{eq:pde_mean} are granted by Theorem \ref{thm:sta2}.



Looking at the proof of Theorem \ref{thm:conv}, we remark that, if the limiting graphon $W$ is deterministic, the initial conditions $\{\theta^i_0\}_{i \in \N}$ can depend on the graph sequence $\{\xi^{(n)}\}_{n\in \N}$. In other words, Theorem \ref{thm:conv} also holds if $\{\theta^i_0\}_{i \in \N}$ is independent of the randomness in $W$ but not necessary on the whole sequence $\{\xi^{(n)}\}_{n\in \N}$. The relationship between the randomness left in $W$ and the randomness in $\xi^{(n)}$ is further discussed in Subsection \ref{ss:examples}.

{\change Requiring that $\Gamma \in \cC^{1+\varepsilon}(\bbT^2)$ for $\varepsilon>0$ has no specific physical meaning. It is a technical assumption to ensure the convergence of estimates involving the Fourier coefficients of $\Gamma$, we refer to the proof of Theorem \ref{thm:conv} in Section \ref{pss:empMes}.}

\subsection{Comparison with the existing literature}
\label{ss:literature}

Weakly interacting particle systems on graph sequences converging to graphons have already been considered in a series of works, both in the stochastic setting \cite{cf:BCW20, cf:L18, cf:RO} and in the deterministic one \cite{cf:CM19, cf:medvedev}. Most of these results {\change vary in terms of the class of particle systems considered, the notion of convergence of the graph sequence or the hypothesis on the initial condition. To our knowledge, no result in the literature considers graph sequences issued from random graphons. In addition, we are not aware of any result directly using the cut-distance as we do.
  
We focus on dense graph sequences where, roughly speaking, the average degree in the graph is proportional to the number of vertices. All the cited results are established in more generality with respect to the graph density, i.e., they consider random graph sequences in different regimes, from dense to almost sparse. There is no common agreement on the terminology used for dense/sparse graphs, we refer to \cite{coppini_weakly_2020} for a presentation of the subject.
  
In the deterministic setting, Medvedev and coauthors, see \cite{cf:CM19, cf:medvedev} and references therein, are the first to consider random graph sequences arising from a graphon. They focus on a restricted class of deterministic models, as the (deterministic) Kuramoto model. In the stochastic setting, \cite{cf:L18} considers particles systems defined on $\R^d$ with additive noise and a transport term in the dynamics, the initial conditions are assumed to be independent but not identically distributed. The hypothesis on the convergence of the graph sequence are somewhat involved as they are not expressed in terms of the cut-norm, the labeled graph limit must satisfy some extra regularity assumption, as continuity or integrability.
The work \cite{cf:RO} focuses on a system of interacting particles on $\R$ where the dynamics is defined by means of a Hamiltonian, the initial conditions are taken to be IID. It establishes a Large Deviation Principle for the empirical measure by explicitly using the cut-norm for the convergence of graph sequences. The proofs are based on Large Deviation techniques. The labeled graph limit must be Lipschitz.

The work \cite{cf:BCW20} considers interacting particle systems on $\R^d$ with multiplicative noise. The graph sequences arise from deterministic graphons under regularity assumptions that are comparable to ours (namely, only measurability for establishing the Law of Large Numbers). They establish a Law of Large Numbers and a Propagation of Chaos property. Their proof is based on careful trajectorial estimates as well as a polynomial representation of the interaction, suitable for the use of the cut-norm. They do not consider random graph limits.

}



\subsection{Applications to exchangeable graphs}
\label{ss:exchangeable}
Recall that  an exchangeable random graph $\xi = \{\xi_{ij}\}_{i,j \in \N}$ ({\change see \cite{cf:DJ08}})  is an infinite array of jointly exchangeable binary random variables, i.e., it satisfies
\begin{equation}
\label{d:excGra}
    \bbP \left( \xi_{ij} = e_{ij}, \, 1 \leq i, j \leq n \right) = \bbP \left( \xi_{ij} = e_{\sigma(i)\sigma(j)}, \, 1 \leq i, j \leq n \right)
\end{equation}
for all $n \in \N$, all permutations $\sigma$ {\change on $\{1, \dots, n\}$} and all $e_{ij} \in \{0,1\}$.

\begin{rem}
    Any finite deterministic graph $\xi$ leads to an exchangeable random graph by performing a uniform random sampling on its associated graphon $W_\xi$, see \eqref{d:graphGraphon} and \cite[\S 10]{cf:Lov}.
        
    More generally, for $W \in \tilde{\cW_0}$ one may construct an exchangeable random graph $\xi^W$, usually called \emph{$W$-random graph}, defined for $i$ and $j$ in $\N$ by
    \begin{equation}
    \label{d:WrandomGraph}
        \xi^W_{ij} = W(U_i, U_j),
    \end{equation}
    where $\{U_i\}_{i \in \N}$ is a sequence of IID uniform random variables on $I$. As recalled by the next theorem, the converse is also true: every exchangeable random graph can be obtained in this way, provided that $W$ is random.
\end{rem}

The characterization of exchangeable random graphs is a consequence of the works of Hoover, Aldous and Kallenberg; see \cite{cf:DJ08} and references therein. We recall their main result here.

\begin{thm}[{\cite[Theorem 5.3]{cf:DJ08}} and {\cite[Theorem 11.52]{cf:Lov}}]
    \label{thm:DJ}
    Let $\xi = \{\xi_{ij}\}_{i,j \in \N}$ be an exchangeable random graph. Then, $\xi$ is a $W$-random graph for some random $W \in \tilde{\cW}_0$. Moreover, let $\xi^{(n)} := \{\xi_{ij}\}_{i,j= 1, \dots, n}$ for every $n \in \N$. It holds that
  \begin{equation}
    \label{eq:graphConv}
        \xi^{(n)} \longrightarrow W \quad \bbP\text{-a.s. in } \tilde{\cW}_0,
    \end{equation}
    as $n\to\infty$.
\end{thm}
We are now ready to state the main corollary of Theorem \ref{thm:conv}, which deals with exchangeable random graphs.
\begin{cor}
    \label{cor:excGra}
    Let $\xi = \{\xi_{ij}\}_{i,j \in \N}$ be an exchangeable random graph and let $W$ be the limit of $\xi^{(n)}:=\{\xi_{ij}\}_{i,j= 1, \dots, n}$ in the sense of Theorem \ref{thm:DJ}. Assume that the initial conditions $\{\theta^i_0\}_{i \in \N}$ are IID and independent of $\{\xi^{(n)}\}_{n \in \N}$, then
    \begin{equation}
        \mu^n \longrightarrow \bar{\mu}, \quad \text{in $\bP \times \bbP$-probability, as $n \to \infty$},
    \end{equation}
    where $\bar{\mu}$ is the solution to \eqref{eq:pde_mean} starting from $\bar{\mu}_0$ with graphon $W$.
\end{cor}

\medskip

We say that $\xi = \{\xi^{(n)}\}_{n\in\N}$ is a sequence of exchangeable graphs when the random variables $\{\xi^{(n)}_{ij}\}_{i,j =1, \dots, n}$ are exchangeable for each $n \in \N$. Observe that $\xi$ is not necessarily an exchangeable random graph as in \eqref{d:excGra}. Whenever $\xi = \{\xi^{(n)}\}_{n\in\N}$ is a sequence of exchangeable graphs, the particles $\{\theta^{i,n}\}_{i =1, \dots, n}$ are exchangeable as well ({\change recall that the initial conditions are assumed IID}) and, in particular, their joint distribution is symmetric, i.e., invariant under permutation of the labels. {\change Observe that this is not true when the graph is not exchangeable.} A classical result by Sznitman \cite[Proposition 2.2]{cf:szni} is that the Law of Large Numbers for the empirical measure of a symmetric joint distribution of particles is equivalent to the propagation of chaos property. From equation \eqref{eq:convergence}, we can thus deduce a propagation of chaos statement for the particle system \eqref{eq:g}. This is illustrated in the next proposition.
\begin{prop}
    \label{p:poc}
    Assume the hypothesis of Theorem \ref{thm:conv}. If $\xi = \{\xi^{(n)}\}_{n\in\N}$ is a sequence of exchangeable graphs, then for every $k \in \N$,
    \begin{equation}
    \lim_{n \to \infty} \cL (\theta^{1,n}, \dots, \theta^{k,n}) = \prod_{i=1}^k \cL(\theta) =  \prod_{i=1}^k  \bar{\mu},
    \end{equation}
    where $\cL$ stands for the law in $\bP \times \bbP$-probability and $\bar{\mu}$ is the solution to \eqref{eq:pde_mean}.
\end{prop}
We omit the proof of Proposition \ref{p:poc}. 


\subsection{Mean-field behavior and two explanatory examples}
\label{ss:examples}
Theorem \ref{thm:conv} allows for a better understanding of the relationship between random graph sequences and the behavior of the empirical measure. More precisely:
\begin{enumerate}
  \item[(1)] It highlights the difference between the randomness present in the graph $\xi^{(n)}$ for every $n \in \N$ and the randomness left in the limit $W$;
  \item[(2)] It presents a new class of random Fokker-Planck equations as possible limit descriptions for the empirical measure $\mu^n$.
\end{enumerate}
As a byproduct, Theorem \ref{thm:conv} yields a precise characterization of the graph sequences for which the empirical measure converges to the mean-field limit. Let us recall what we mean by mean-field limit and first discuss this last issue; we then address (1) and (2) with the help of two examples.

Consider system \eqref{eq:g} on a sequence of complete graphs, i.e., $\xi^{(n)}_{ij} \equiv 1$ for every $i, \ j$ and $n$. It is well known \cite{cf:Oel84,cf:szni} that the empirical measure $\mu^n$ converges to the \emph{mean-field limit} $\rho \in \cC([0,T],\cP(\bbT))$, defined as the unique solution to the following McKean-Vlasov equation:
\begin{equation}
\label{eq:McKeanVlasov}
    \partial_t \rho_t (\theta) = \frac 12 \partial^2_\theta \rho_t (\theta) - \partial_\theta \left[\rho_t(\theta) F(\theta) \right] - p \, \partial_\theta \left[\rho_t (\theta) \int_\bbT \Gamma(\theta, \tilde{\theta}) \, \rho_t(\dd \tilde{\theta}) \right],
\end{equation}
with initial condition $\bar{\mu}_0$ and $p=1$. Existence and uniqueness for the solution to \eqref{eq:McKeanVlasov} hold under our assumptions on $F, \, \Gamma$ and $\bar{\mu}_0$, see again \cite{cf:Oel84, cf:szni}.

Suppose that the graph sequence is converging to a deterministic limit; we discuss the case of a random limit after the first example. Theorem \ref{thm:conv} implies that for every sequence $\{\xi^{(n)}\}_{n\in\N}$ which converges to some flat graphon $W\equiv p \in [0,1]$, the empirical measure $\mu^n$ {\change in the limit} satisfies equation \eqref{eq:McKeanVlasov} with corresponding $p$. Since the convergence of $\xi^{(n)}$ to a non-constant graphon gives rise to equation \eqref{eq:pde_mean}, which is \textendash{} at least formally \textendash{} different from \eqref{eq:McKeanVlasov}, {\change we conclude that, if the sequence $\xi^{(n)}$ converges to a constant graphon, then the limit of $\mu^n$ is formally mean-field}. The graphs with such asymptotic behavior are known in the literature as pseudo-random graphs, see \cite{cf:BBCN16, cf:CGW89} and \cite[\S 11.8.1]{cf:Lov}.

\medskip

We now address the issues (1) and (2) with two explanatory examples. The mean-field comparison when the graph limit is random is discussed after the next example.

\subsubsection{Example I: $W$-random graphs}
Fix $p \in (0,1)$ and let $g$ be a random variable on $(0,1)$ with mean $\sqrt{p}$ and distribution function given by $F_g$. Let $\{g_i\}_{i \in \N}$ be a sequence of IID copies of $g$. Conditionally on $\{g_i\}_{i\in \N}$, $\xi^{(n)}_{ij}$ is defined as
\begin{equation}
\label{ex:ber}
\xi^{(n)}_{ij} \sim \text{Ber}(g_i g_j), \quad \text{independently for each } 1 \leq i < j \leq n.
\end{equation}
The graph $\xi^{(n)}$ is the dense analogue of the inhomogeneous random graph, also known as rank-1 model, see e.g., \cite{cf:BvdHvL,cf:BCvdH}. In this model, $g_i$ corresponds to the weight associated with particle $i$ and, loosely speaking, the closer $g_i$ is to 1, the more connections particle $i$ forms. We expect that assigning different distributions to $g$ leads to different behaviors for the empirical measure \eqref{d:emp}.

The construction made in \eqref{ex:ber} yields a binary array $\{\xi^{(n)}_{ij}\}_{i,j = 1, \dots, n}$ of exchangeable random variables. In particular, {\change the edges} have the same expected value $\bbE[\xi^{(n)}_{ij}] = p$, for every $i\neq j$. We are interested in comparing the empirical measure of the system \eqref{eq:g} defined on the graph \eqref{ex:ber} to the empirical measure of the corresponding \emph{annealed} system that is obtained from \eqref{eq:g} by replacing $\xi^{(n)}_{ij}$ with their expected value. More precisely, the annealed system is defined as the solution to
\begin{equation}
\label{eq:gC}
\dd{\theta}_{t}^{i,n} =\, F(\theta^{i,n}_t) \dd t +  \frac{p}{n} \sum_{j=1}^{n}  \Gamma (\theta_{t}^{i,n}, \theta_{t}^{j,n}) \dd t + \dd B_{t}^{i},
\end{equation}
for which the asymptotic behavior is known to be the mean-field limit \eqref{eq:McKeanVlasov}.

%
{\change The behavior of system \eqref{eq:g} on the graph sequence} \eqref{ex:ber} is described in the limit by \eqref{eq:gC} only when $g$ is deterministic and equal to $\sqrt{p}$. Recall the definition of $W$-random graph given in \eqref{d:WrandomGraph}: we see that $\xi^{(n)}$ is a $W_g$-random graph with
\begin{equation}
    W_g(x,y) = F^{-1}_g (x) \, F^{-1}_g(y), \quad  \text{for } x, y \in I,
\end{equation}
where $F^{-1}_g$ is the pseudo inverse of $F_g$. In particular, the $\bbP$-a.s. limit of $\xi^{(n)}$ is given by $W_g$ and thus the limit of $\mu^n$ by the solution to equation \eqref{eq:pde_mean} with $W=W_g$. Theorem \ref{thm:conv} and Theorem \ref{thm:sta2} imply that if $W_g$ is  arbitrarily close to the constant graphon $p$ in the cut-distance, i.e., if  $\text{Var}[g] \ll 1$, then the empirical measure of the system associated to $\xi^{(n)}$ is arbitrarily close to the mean-field limit of the annealed system \eqref{eq:gC}. In this case, $\{\xi^{(n)}\}_{n \in \N}$ is close to an Erd\H{o}s-Rényi graph sequence, for which the mean-field behavior is already known, see \cite{cf:CDG}. Finally, observe that choosing a suitable deterministic sequence of the weights $\{g_i\}_{i \in [n]}$, e.g., $g_i = F^{-1}_g (i/n)$ for $i \in [n]$, would lead to a random graph $\xi^{(n)}$ which is not exchangeable. In particular, $\bbE[\xi^{(n)}_{ij} ]$ is not constant and changes for every $i$ and $j$. Nonetheless, the sequence $\xi^{(n)}$ still converges to the same limit $W_g$, $\bbP$-a.s.~in the realization of the Bernoulli random variables and possibly at the cost of requiring some regularity on $W_g$, see \cite[\S 11.4]{cf:Lov}.

This example illustrates how the randomness related to the exchangeability in the sequence $\xi^{(n)}$ is lost in the limit of $\mu^n$, as it is lost in the graph limit $W_g$. In this sense, adding exchangeability to system \eqref{eq:g} does not yield any \emph{averaging property} on the empirical measure $\mu^n$. Moreover, adding the extra randomness through Bernoulli random variables in \eqref{ex:ber} does not alter this fact. In other words, taking $\xi^{(n)}_{ij} = g_i g_j \in [0,1]$ yields yet again the same limit for $\mu^n$.

Until now, we have focused on deterministic limits for the sequence $\xi^{(n)}$.
We now consider the case where the limit $W$ is random, and we address the relationship between the resulting system and the mean-field limit $\rho$ given in \eqref{eq:McKeanVlasov}. One might be led to conjecture that it is possible to recover the mean-field behavior by, e.g., averaging the limit dynamics with respect to the randomness in $W$. In the next example, we formulate this remark in a rigorous way. We show that this is in general not possible, although it may lead to a new class of asymptotic behaviors which are interesting on their own, as pointed out in the bullet point (2) above.

\subsubsection{Example II: random mean-field behavior}
Consider the growing preferential attachment graph $\xi_{\text{pa}}$ constructed iteratively as follows; see also \cite[Example 11.44]{cf:Lov}. Begin with a single node and, assuming that at the $n$-th step there are already $n$ nodes, create a new node with label $n+1$ and connect it to each node $i\in\{1,\ldots,n\}$ with probability $(d_n(i)+1)/(n+1)$ where $d_n(i)$ is the degree of node $i$ at step $n$ and each connection is made independently of the others. Denote the corresponding random graph by $\xi^{(n+1)}_{\text{pa}}$.

Roughly speaking, the behavior of $\xi_{\text{pa}}$ depends crucially on the first steps of the construction and it stabilizes to a homogeneous structure as $n$ grows. This is illustrated in the next proposition.

\begin{prop}[{\cite[Proposition 11.45]{cf:Lov}}]
  \label{p:GPA}
  With probability 1, the sequence $\{\xi^{(n)}_{\emph{pa}}\}_{n \in \N}$ converges to a random constant graphon.
\end{prop}

Consider a particle system defined on the graph sequence $\{\xi^{(n)}_{\text{pa}}\}_{n \in \N}$. The empirical measure converges to the solution of equation \eqref{eq:McKeanVlasov} with a random $p$. In other words, $\mu^n$ converges to a random mean-field limit. Integrating \eqref{eq:McKeanVlasov} with respect to this randomness and denoting $\bbE[\rho_t]$ by $\bar{\rho}_t$ for every $t \in [0,T]$, we obtain that $\bar{\rho} \in \cC([0,T],\cP(\bbT))$ satisfies
\begin{equation}
\label{eq:randomMcKeanVlasov}
\partial_t \bar{\rho}_t (\theta) = \frac 12 \partial^2_\theta \bar{\rho}_t (\theta) - \partial_\theta \left[\bar{\rho}_t(\theta) F(\theta) \right] - \partial_\theta \left[ \bbE \Big[ p \, \rho_t (\theta) \int_\bbT \Gamma(\theta, \tilde{\theta}) \rho_t(\dd \tilde{\theta}) \Big]\right],
\end{equation}
for $t \in [0,T]$. Note that \eqref{eq:randomMcKeanVlasov} is not written in closed form because of the third term on the right-hand side which is not linear in $\rho$ and $p$. In this sense, $\bar{\rho}$ does not formally satisfy the mean-field limit, i.e., it is not a solution to \eqref{eq:McKeanVlasov} with some deterministic $p \in [0,1]$. {\change By definition, $\bar{\rho}$ is a mixture of mean-field limits, weighted by the distribution of $p$.}


To have an intuitive understanding of what $\bar{\rho}$ may look like, consider the stochastic Kuramoto model without natural frequencies \cite{cf:BGP10,cf:C19} defined on the sequence $\xi^{(n)}_{\text{pa}}$. The model is defined as the solution to
\begin{equation}%
\label{eq:kur}
\dd \theta^{i,n}_t = \frac Kn \sum_{j=1}^n \xi^{(n)}_{ij} \sin (\theta^{j,n}_t-\theta^{i,n}_t) \dd t + \dd B^i_t,
\end{equation}%
for $i =1, \dots, n$ and $t \in [0,T]$. It corresponds to \eqref{eq:g} with the choices $F\equiv 0$ and $\Gamma(\theta,\psi) = -K\sin(\theta-\psi)$. An application of Theorem \ref{thm:conv} and Proposition \ref{p:GPA} implies that the empirical measure of \eqref{eq:kur} converges to the solution of
\begin{equation}
\label{eq:kurC}
\partial_t \rho_t (\theta) = \frac 12 \partial^2_\theta \rho_t(\theta) + pK \partial_\theta [ \rho_t(\theta) (\sin * \rho_t) (\theta)],
\end{equation}
where $*$ stands for the convolution operator. It is well-known that the system \eqref{eq:kurC} undergoes a phase transition as the coupling strength $pK$ crosses the critical threshold $pK=1$. Hence, the phase transition for this specific model occurs at a \textit{random} critical threshold. Depending on the sampled value of $p$, one obtains stable synchronous solutions in the supercritical regime ($pK >1$), or uniformly distributed oscillators on $\bbT$ ($0\leq pK<1$). The solution to equation \eqref{eq:kurC} can be written down explicitly (see again \cite{cf:BGP10,cf:C19}) and, integrating over the randomness of $p$, gives a superposition of synchronous and asynchronous states which, in general, is \textit{not} a mean-field solution, i.e., it does not solve \eqref{eq:kurC} for some fixed $p \in [0,1]$.



\section{The non-linear process}
\label{s:nonLin}
We introduce a non-linear process \eqref{eq:nonLin1} which has already been considered in the literature \cite{cf:BCW20, cf:CM19, cf:L18, cf:LS14,cf:RO} as the natural candidate in case the particles in \eqref{eq:g} are not exchangeable and their labels are fixed from the initial condition. This process is useful for studying the evolution of a tagged particle with a specific profile of connections, as stressed in \cite{cf:L18}.

Contrary to our setting, some regularity in the \textendash{} now labeled \textendash{} graphon is usually assumed to show the convergence of the empirical measure \eqref{d:emp}. We will exploit {\change the non-linear process with fixed labels} \eqref{eq:nonLin1} to better understand \eqref{eq:nonLin2} and to establish existence and uniqueness.
    

Before introducing \eqref{eq:nonLin1}, we define some other tools for dealing with empirical measures and graphons. Notably, we introduce an equivalence relation between probability measures on $I \times \bbT$ inspired by graph limits theory, see \eqref{d:equiv}. This will allow us to prove Theorem \ref{thm:sta2}, where we establish that the empirical measure is H\"older continuous with respect to the underlying graphon.

\subsection{Distances between probability measures}
Let $\cM_T$ be the space of probability measures on $I \times \cC([0,T], \bbT)$ with first marginal equal to the Lebesgue measure $\lambda$ on $I$, i.e.,
\begin{equation}
\cM_T := \left\{ \mu \in \cP(I \times \cC([0,T], \bbT)) : p_1 \circ \mu = \lambda \right\},
\end{equation}
where $p_1$ is the projection map associated to the first coordinate. For $\mu \in \cM_T$ the following decomposition holds
\begin{equation}
\mu(\dd x, \dd \theta) = \mu^x (\dd\theta) \lambda(\dd x), \quad x\in I,
\end{equation}
where $\mu^x \in \cP(\cC([0,T], \bbT))$ for almost every $x \in I$. From now on, we denote the Lebesgue measure $\lambda (\dd x)$ on $I$ simply by $\dd x$. For $\mu, \nu \in \cM_T$, define their distance by
\begin{equation}\label{d:extended_distance}
d_T (\mu,\nu) := \left(\int_I D^2_T (\mu^x,\nu^x) \dd x \right)^{1/2}.
\end{equation}

\begin{rem}
\label{rem:measurability}
{\change 
  Observe that the previous expression makes sense as $I\ni x \mapsto \mu^x \in \cP(\cC([0,T], \bbT))$ is a $\lambda$-measurable function. This is because we consider $\mu \in \cP(I \times \cC([0,T], \bbT))$. If one considers $\cP(\cC([0,T], \bbT))^I$, as done in \cite{cf:BCW20}, then an extra assumption is required, namely that every $\eta = \{\eta^x\}_{x \in I} \in \cP(\cC([0,T], \bbT))^I$ is $\lambda$-measurable in $x \in I$.
}
\end{rem}
%

\medskip

\begin{rem}
Observe that the previous definitions make sense also with $T=0$ and $\cC([0,T], \bbT)$ replaced by $\bbT$. In particular, $\cM_0$ is the space of probability measures on $I\times \bbT$ with first marginal equal to the Lebesgue measure $\lambda$ on $I$, i.e.
\begin{equation}
\label{d:M_0}
\cM_0 = \left\{ \mu_0 \in \cP(I \times \bbT) : p_1 \circ \mu_0 = \lambda \right\},
\end{equation}
and
\begin{equation}
d_0 (\mu_0,\nu_0) = \left(\int_I D^2_0 (\mu^x_0,\nu^x_0) \, \dd x\right)^{1/2}, \quad \text{ for }\mu_0, \nu_0 \in \cM_0.
\end{equation}
\end{rem}


Inspired by the graphon framework, one can define the following relation of equivalence on $\cM_T$ (the case $T=0$ is analogous): for $\mu, \nu \in \cM_T$
\begin{equation}
\label{d:equiv}
\mu \sim \nu \quad \text{ iff there exists $\varphi \in S_I$ such that }  \mu^x = \nu^{\varphi(x)}, \; x\text{-a.s.}.
\end{equation}
Endow the quotient space $\cM_T / \sim$ with the induced distance given by
\begin{equation}
\label{d:normQ}
\tilde{d}_T (\mu, \nu) := \inf_{\varphi \in S_I} d_T (\mu, \nu^\varphi),
\end{equation}
where we have used the notation $\nu^\varphi = \{\nu^{\varphi(x)}\}_{x \in I}$. Observe that if $\mu \sim \nu$, then $\bar{\mu} = \int_I \mu^x \dd x = \int_I \nu^{\varphi(x)} \dd x = \int_I \nu^x \dd x = \bar{\nu}$. In particular, for every $\varphi \in S_I$
\begin{equation}
D^2_T(\bar{\mu}, \bar{\nu}) = D^2_T(\bar{\mu}, \bar{\nu}^\varphi )\leq \int_I D^2_T (\mu^x, \nu^{\varphi(x)}) \dd x = d^2_T (\mu, \nu^\varphi).
\end{equation}
By taking the infimum with respect to $\varphi \in S_I$, we obtain
\begin{equation}
\label{eq:dis_inequality}
D_T(\bar{\mu}, \bar{\nu}) \leq \tilde{d}_T (\mu, \nu).
\end{equation}



\subsection{The non-linear process with fixed labels}
Fix a labeled graphon $W \in \cW_0$ together with an initial condition $\mu_0 \in \cM_0$. Consider the process $\theta = \{\theta^x\}_{x \in I}$ that solves the system
\begin{equation}
\label{eq:nonLin1}
\begin{cases}
\theta^x_t = \theta^x_0 + \int_0^t F(\theta^x_s) \dd s + \int_0^t \int_I W(x,y) \int_\bbT  \Gamma(\theta^x_s, \theta) \mu^y_s(\dd \theta) \dd y \, \dd s + B^x_t, \\
\mu^x_t = \cL(\theta^x_t), \quad \text{ for } x\in I, \, t \in [0,T],
\end{cases}
\end{equation}
where $\{\theta^x_0\}_{x \in I}$ is a random vector such that $\cL(\theta^x_0)= \mu^x_0$ for $x \in I$ and $\{B^x\}_{x\in I}$ a sequence of IID Brownian motions independent of  $\{\theta^x_0\}_{x \in I}$. The following proposition shows existence and uniqueness for the solution of \eqref{eq:nonLin1}. The proof follows a classical argument by Sznitman \cite{cf:szni}  and is postponed to Section \ref{pss:nonLin1}.

\begin{prop}
    \label{p:nonLin1}
    There exists a unique {\change pathwise} solution $\theta = \{\theta^x\}_{x \in I}$ to \eqref{eq:nonLin1}. {\change If $\nu^x \in \cC([0,T],\cP(\bbT))$ denotes the law of $\theta^x$ for $x \in I$, then $I \ni x \mapsto \nu^x \in \cP( \cC([0,T], \bbT))$ is Lebesgue measurable. For every $x \in I$,} $\nu^x$ satisfies the following non-linear Fokker-Planck equation in the weak sense
    \begin{equation}
    \label{eq:pde}
    \partial_t \mu^x_t (\theta) = \frac 12 \partial^2_\theta \mu^x_t(\theta) - \partial_\theta \left[\mu^x_t(\theta) F(\theta) \right] - \partial_\theta \left[\mu^x_t(\theta) \int_I W(x,y) \int_\bbT \Gamma(\theta, \theta')\mu^y_t(\dd \theta') \dd y  \right]
    \end{equation}
    with initial condition $\mu^x_0 \in \cP(\bbT)$.
\end{prop}

The process $\{\theta^x\}_{x \in I}$ is indexed by the space of labels $I$. For two different labels $x$ and $y$ in $I$, the behavior of particles $\theta^x$ and $\theta^y$ may vary depending on their connection profile encoded in $W$ and the two marginals $\mu^x$ and $\mu^y$ may vary as well. Similar results in different settings have already been shown in \cite{cf:BCW20, cf:CH18, cf:L18, cf:LS14, cf:RO}.

It is interesting to know that the law $\mu=\{\mu^x\}_{x \in I} \in \cM_T$ is continuous with respect to the cut-norm (or equivalently in $d_\square$-distance) in $\cW_0$, as already remarked in \cite[Theorem 2.1]{cf:BCW20} for much more general systems than the ones we consider here. Exploiting the compactness of $\bbT$ and some extra regularity of $\Gamma$, we are able to prove that the map $W\mapsto \mu^W$ is Hölder-continuous, as shown in the next proposition.

\begin{prop}
    \label{p:sta1}
    Assume the hypothesis of Theorem \ref{thm:conv}. There exists a positive constant $C$, {\change only depending on $\Gamma$ and $T$}, such that, if $\mu^W$ and $\mu^V$ denote the laws of the solutions to \eqref{eq:nonLin1} with $W \in \cW_0$ and $V \in \cW_0$ respectively, then
    \begin{equation}
    d_T(\mu^W, \mu^V) \leq C \norm{W-V}^{1/2}_{\square}.
    \end{equation}
\end{prop}

{\change The proof is based on classical trajectorial estimates and Fourier analysis. As $\Gamma$ is a function on the torus, it is possible to apply Fourier-type arguments to factorize it in its two components, this is the key point to make the graphon norm $\norm{\cdot}$ appear and prove Proposition \ref{p:sta1}}. The full proof is postponed to Subsection \ref{pss:nonLin1}.

The previous proposition can be traduced in terms of the cut-distance $\delta_\square$, recall \eqref{d:cut-distance}, and the space of graphons $\tilde{\cW}_0$.

\begin{thm}
    \label{thm:sta2}
    Assume the hypothesis of Theorem \ref{thm:conv}. Then, there exists a positive constant $C$, {\change only depending on $\Gamma$ and $T$}, such that, if $\bar{\mu}^W$ and $\bar{\mu}^V$ denote the laws of the solutions to equation \eqref{eq:nonLin2} associated with graphons $W$ and $V$ respectively, it holds that
    \begin{equation}
        \label{eq:holderContinuity}
        D_T(\bar{\mu}^W, \bar{\mu}^V) \leq C \, {\delta_{\square} (W,V)}^{1/2}.
    \end{equation}
\end{thm}%
{\change The proof of Theorem \ref{thm:sta2} is a straightforward consequence of Proposition \ref{p:sta1}, it is given in Subsection \ref{pss:nonLin2}.}
  
The 2-Wasserstein distance $D_T$, recall the definition \eqref{d:wass_T}, could be replaced with the $p$-Wasserstein distance in \eqref{d:wass_T} for $p \geq 1$. This would lead to a H\"older exponent in \eqref{eq:holderContinuity} as large as $1/p$. We stick to $p=2$ but any choice could be possible, modulo a different constant $C$.

Observe that Theorem \ref{thm:nonLin2} and Theorem \ref{thm:sta2} imply that the following mapping is continuous:
\begin{equation}
    \begin{split}
        \Psi : (\tilde{\cW}_0, \delta_\square) &\rightarrow ( \cC([0,T], \cP(\bbT)), D_T) \\
        W & \mapsto \bar{\mu}^W,
    \end{split}
\end{equation}
where $\bar{\mu}^W$ is the law of $\theta$ solving equation \eqref{eq:nonLin2} with graphon $W$. In particular, to every random variable $W$ in $\tilde{\cW}_0$ corresponds a random variable $\bar{\mu}^W$ with values in $\cC([0,T], \cP(\bbT))$, i.e., for almost every $\omega \in \Omega$, $\bar{\mu}^W (\omega) = \bar{\mu}^{W(\omega)}$.

\begin{rem}
    We point out that Theorem \ref{thm:sta2} allows to conclude that two solutions to equation \eqref{eq:pde_mean} are close as probability measures if the corresponding graphons are close in $\tilde{\cW}_0$. However, whether two different graphons can lead to similar behaviors in the particle system is still not clear. In other words, we are not able to provide any lower bound complementary to the upper bound given in equation \eqref{eq:holderContinuity}. To our knowledge, this aspect may be model-dependent and needs further investigations.
\end{rem}

\subsubsection{Relation between label and unlabeled non-linear processes}
Consider a probability distribution $\mu_0 \in \cM_0$ such that $\int_I \mu^x_0 \, \dd x = \bar{\mu}_0$. The solution to \eqref{eq:pde_mean} is given by $\bar{\mu} = \int_I \mu^x \, \dd x$, where $\mu^x$ is the law of $\theta^x$ solving \eqref{eq:nonLin1} with  initial condition $\mu^x_0$ and labeled graphon $W$. In other words, $\theta$ has the same law as $\theta^U$ solution to \eqref{eq:nonLin1}, where $U$ is a uniform random variable in $I$ independent of the other randomness in the system. As the following remark shows, the law $\bar{\mu}$ of $\theta$ {\change does neither depend on the representative $W$, nor} on $\mu_0$.

\begin{rem}
    \label{rem:varphi}
    Let $\varphi \in S_I$, i.e., $\varphi$ is an invertible measure preserving map from $I$ to itself, and $\nu = \{ \nu^x\}_{x \in I}$ the law of $\{\theta^{\varphi(x)}\}_{x \in I}$ solving \eqref{eq:nonLin1}. By a change of variables, $\theta^{\varphi(x)}$ solves
    \begin{equation}
    \theta^{\varphi(x)}_t = \theta^{\varphi(x)}_0 + \int_0^t F(\theta^{\varphi(x)}_s) \dd s +  \int_0^t \int_I W({\varphi(x)},\varphi(y)) \int_\bbT \Gamma(\theta^{\varphi(x)}_s, \theta) \mu^{\varphi(y)}_s(\dd \theta) \dd y \, \dd s + B^{\varphi(x)}_t
    \end{equation}
    and can be rewritten with $V= W^\varphi$  and $\psi^x = \theta^{\varphi(x)}$ as
    \begin{equation}
    \label{eq:nonLin1phi}
    \psi^{x}_t = \theta^{\varphi(x)}_0 +  \int_0^t F(\psi^x_s) \dd s  + \int_0^t \int_I V(x,y) \int_\bbT \Gamma(\psi^x_s, \theta) \nu^y_s(\dd \theta) \dd y \, \dd s + B^{\varphi(x)}_t,
    \end{equation}
    which has the same law as \eqref{eq:nonLin1} with labeled graphon $V$ and initial conditions $\{\theta^{\varphi(x)}\}_{x\in I}$.
    
    Observe that the laws $\nu$ and $\mu$ associated to \eqref{eq:nonLin1phi} and \eqref{eq:nonLin1} respectively, differ only in the labeling of the vertices but their distance in $\cM_T$ is not zero due to the initial conditions and the fact that $\norm{W-V}_\square = \norm{W-W^\varphi}_\square$ is, in general, different from zero. However, if one looks at $\bar{\mu} = \int_I \mu^x \, \dd x$ and $\bar{\nu} = \int_I \nu^x \, \dd x$, they coincide as probability measures in the sense that $D_T(\bar{\mu},\bar{\nu})=0$. In particular, the law of the solution to equation \eqref{eq:nonLin2} is also equivalent to $\psi^U$, where $\psi^x$ solves \eqref{eq:nonLin1phi}, and $U$ is uniformly distributed on $I$.
\end{rem}


\subsection{Proofs for the non-linear process \eqref{eq:nonLin1} with fixed labels}
\label{pss:nonLin1}

\begin{proof}[Proof of Proposition \ref{p:nonLin1}]
    The proof follows a classical argument given in \cite[Lemma 1.3]{cf:szni}. Consider $\nu \in \cM_T$ and $\{\theta^{x,\nu}\}_{x \in I}$ solving
    \begin{equation}
    \label{nl:linearEq}
        \theta^{x, \nu}_t = \theta^x_0 + \int_0^t F(\theta^{x,\nu}_s) \, \dd s + \int_0^t \int_I W(x,y) \int_\bbT \Gamma (\theta^{x,\nu}_s, \theta) \nu^y_s (\dd \theta) \, \dd y \, \dd s + B^x_t,
    \end{equation}
    where the initial conditions and the Brownian motions are the same of \eqref{eq:nonLin1}. {\change Observe that, by Remark \ref{rem:measurability}, the mapping $I \ni x \mapsto \nu^x_s \in \cP(\bbT)$ is Lebesgue measurable for every $s \in [0,T]$}. Since $F$ and $\Gamma$ are bounded Lipschitz functions, there exists a unique solution to \eqref{nl:linearEq}, which we denote by $\Phi(\nu) \in \cM_T$. The solution of \eqref{nl:linearEq} is constructed as the limit of a Cauchy sequence of elements of $\cM_T$. Since these are measurable as functions of $x\in I$, the mapping $I \ni x \mapsto \text{Law}(\theta^{x,\nu})\in\cP(\cC([0,T], \bbT))$ is also Lebesgue measurable. Thus, the map
    \begin{equation}
    \begin{split}
        \Phi : (\cM_T, d_T) &\to (\cM_T, d_T) \\
                \nu & \to \Phi(\nu)
    \end{split}
    \end{equation}
    is well defined. A solution to \eqref{eq:nonLin1} is a fixed point of $\Phi$ and any fixed point of $\Phi$ is a solution to \eqref{eq:nonLin1}.
    
    For $\mu, \nu \in \cM_T$, consider the processes $\theta^{x,\mu}$ and $\theta^{x,\nu}$, with $x \in I$. We estimate their distance as
    \begin{align}
         &\left| \theta^{x,\mu}_t - \theta^{x,\nu}_t \right|^2 \leq C \int_0^t \left|F(\theta^{x,\mu}_s) - F(\theta^{x,\nu}_s)\right|^2 \dd s \notag\\
        & + C \int_0^t \left| \int_I W(x,y) \left(\int_\bbT \Gamma (\theta^{x,\mu}_s, \theta) \mu^y_s (\dd \theta) - \int_\bbT \Gamma (\theta^{x,\nu}_s, \theta) \nu^y_s (\dd \theta)\right) \dd y  \right|^2 \dd s \notag
\end{align}

Adding and subtracting in the second integral the quantity $\Gamma (\theta^{x,\mu}_s, \theta) \nu^y_s(\dd \theta)$
and using that $F$ and $\Gamma$ are Lipschitz-continuous functions and that $F, \, \Gamma$ and $W$ are bounded, we get
    \begin{equation}
 \leq C\int_0^t \left| \theta^{x,\mu}_s - \theta^{x,\nu}_s \right|^2 \dd s + C\int_0^t \int_I \left|\int_\bbT \Gamma (\theta^{x,\mu}_s, \theta) \left[ \mu^y_s -\nu^y_s \right] (\dd \theta) \right|^2 \dd y \, \dd s,
    \end{equation}

     From \eqref{d:lips_dist} we obtain
    \begin{equation}
    \left|\int_\bbT \Gamma (\theta^{x,\mu}_s, \theta) \left( \mu^y_s -\nu^y_s \right) (\dd \theta) \right| \leq D_s (\mu^y, \nu^y)
    \end{equation}
    from which, using \eqref{d:extended_distance}, we deduce
    \begin{equation}
        \left| \theta^{x,\mu}_t - \theta^{x,\nu}_t \right|^2 \leq C\int_0^t \left| \theta^{x,\mu}_s - \theta^{x,\nu}_s \right|^2 \dd s + C\int_0^t d^2_s (\mu, \nu) \, \dd s.
    \end{equation}
    The definition of $D_T$ \eqref{d:wass_T2} and an application of Gronwall's lemma lead to
    \begin{equation}
    \label{nl:map}
    d^2_T(\Phi (\mu), \Phi (\nu)) \leq \int_I \bE \left[ \sup_{t \in [0,T]} \left| \theta^{x,\mu}_t - \theta^{x,\nu}_t \right|^2 \right] \dd x \leq C \int_0^T d^2_s (\mu, \nu) \, \dd s.
    \end{equation}
    From the last relation we obtain the uniqueness of solutions to \eqref{eq:nonLin1}.
    
    We prove that a solution exists by iterating \eqref{nl:map}. Indeed, for $k \geq 1$ and $\mu \in \cM_T$, one gets
    \begin{equation}
        d^2_T(\Phi^{k+1} (\mu), \Phi^k (\mu)) \leq C^k \frac {T^k}{k!} \int_0^T d^2_t(\Phi(\mu), \mu) \, \dd t.
    \end{equation}
    In particular, $\{\Phi^k (\mu)\}_{k \in \N}$ is a Cauchy sequence for $k$ large enough, and its limit is the fixed point of $\Phi$. Note that $d_t(\Phi(\mu), \mu) < \infty$ since we are working on the compact space $\bbT$.
    
    \medskip
    
    For the second part of Proposition \ref{p:nonLin1}, apply It\^o's formula to $f(\theta^x_t)$ with $f \in \cC^\infty_0$ to get
    \begin{equation}
    \begin{split}
        f(\theta^x_t) & \, =\,  f(\theta^x_0) + \frac 12  \int_0^t \partial^2_\theta f(\theta^x_s) \, \dd s + \int_0^t \partial_\theta f(\theta^x_s) \, F(\theta^x_s)  \, \dd s \\
        + & \int_0^t \partial_\theta f(\theta^x_s) \, \left[ \int_I W(x,y) \int_\bbT \Gamma (\theta^x_s, \theta) \mu^y_s (\dd \theta) \dd y \right] \dd s + \int_0^t \partial_\theta f(\theta^x_s) \, \dd B^x_s.
    \end{split}
    \end{equation}
    Integrating with respect to $\bP$ yields the weak formulation of \eqref{eq:pde}.
\end{proof}


Next we move to the proof of Proposition \ref{p:sta1}.

\begin{proof}[Proof of Proposition \ref{p:sta1}]
    Let $\{\theta^{x,W}\}_{x \in I}$ and $\{\theta^{x,V}\}_{x \in I}$ be the two non-linear processes associated to $W$ and $V$ respectively. We compare the two solutions: as done in the proof of Proposition \ref{p:nonLin1}, by adding and subtracting in the integrals the term $W(x,y) \Gamma(\theta^{x,V}_r, \theta) (\mu^{y,W}_r-\mu^{y,V}_r)$ we get
  \begin{equation}
    \begin{split}
        &\left|\theta^{x,W}_s - \theta^{x,V}_s \right|^2 \leq C \int_0^s \left| F(\theta^{x,W}_r) - F(\theta^{x,V}_r ) \right|^2 \dd r \\
        & + C \int_0^s \left|\int_I W(x,y) \int_\bbT (\Gamma(\theta^{x,W}_r, \theta)- \Gamma(\theta^{x,V}_r, \theta))  \mu^{y,W}_r (\dd \theta) \, \dd y \right|^2 \dd r  \\
        & + C \int_0^s \left|\int_I W(x,y) \int_\bbT \Gamma(\theta^{x,V}_r, \theta) (\mu^{y,W}_r - \mu^{y,V}_r) (\dd \theta) \, \dd y \right|^2 \dd r  \\
        & + C \int_0^s \left| \int_I (W(x,y) - V(x,y)) \int_\bbT \Gamma(\theta^{x,V}_r, \theta) \mu^{y,V}_r (\dd \theta) \, \dd y \right|^2 \dd r.
    \end{split}
    \end{equation}
Using that $F$ and $\Gamma$ are Lipschitz-continuous functions and that $F, \, \Gamma$ and $W$ are bounded, we get
    \begin{align}
        \left|\theta^{x,W}_s - \theta^{x,V}_s\right|^2 \leq C\int_0^s \left|\theta^{x,W}_r - \theta^{x,V}_r \right|^2 \dd r + C \int_0^s d^2_r(\mu^{W}, \mu^{V}) \, \dd r \notag\\
        + \int_0^s \left| \int_I \left(W(x,y) - V(x,y) \right) \left( \int_\bbT \Gamma(\theta^{x,V}_s,\theta) \mu^{y,V}_r (\dd \theta)\right) \dd y \right|^2 \dd r.
    \end{align}
    After taking the supremum over $s\in[0,t]$, the expectation $\bE$ and integrating with respect to $x \in I$, we are able to apply Gronwall's lemma as in \eqref{nl:map} to get
    \begin{equation}
  \label{eq:proof_Holder_first_Gronwall}
        d^2_t(\mu^W, \mu^V) \leq \int_I \bE \left[ \sup_{s \in [0,t]} \left| \theta^{x,W}_s - \theta^{x,V}_s \right|^2 \right] \dd x \leq C \Big( \int_0^t d^2_s (\mu^W, \mu^V) \, \dd s + G \Big) ,
    \end{equation}
    where $G$ is given by
    \begin{equation}
    \label{eq:G}
    G = \int_0^t \bE \left[\int_I \left| \int_I \left(W(x,y) - V(x,y) \right) \left( \int_\bbT \Gamma(\theta^{x,V}_s,\theta) \mu^{y,V}_s (\dd \theta)\right) \dd y \right|^2 \dd x \right] \dd s.
    \end{equation}
    Applying Gronwall's inequality to \eqref{eq:proof_Holder_first_Gronwall} yields
    \begin{equation}
        d^2_t(\mu^W,\mu^V) \leq C G.
    \end{equation}
    The proof is concluded provided that $G \leq C' \norm{W-V}_\square$, for some constant $C'>0$.
    
    {\change We take advantage of the Fourier series representation of $\Gamma$. As we are working on the torus, we can factorize $\Gamma$ in its two components}, i.e.,
    \begin{equation}
    \label{gammaFourier}
        \Gamma (\theta, \psi) = \sum_{k,l \in \Z} \Gamma_{kl} \, e^{ik \theta} e^{il \psi}, \quad \theta, \psi \in \bbT,
    \end{equation}
    where $\Gamma_{kl} = \int_{\bbT^2} \Gamma(\theta, \psi) e^{i (k \theta +l \psi)} \dd \theta \dd \psi$. {\change The condition $\Gamma \in \cC^{1+\varepsilon}$ allows to have a decay estimate on the Fourier coefficients $\Gamma_{kl}$. Indeed,} classical results on the asymptotic of Fourier series \cite[pp. 24-26]{Katznelson} imply that
    \begin{equation}
    \label{gammaAbsConv}
        C_\Gamma := \sum_{k,l \in \Z} (kl)^{1+\varepsilon} \left| \Gamma_{kl} \right|^2< \infty.
    \end{equation}
    Plugging this expression into \eqref{eq:G}, we obtain that
    \begin{equation}
    \begin{split}
        &\int_I \left|\int_I \left(W(x,y) - V(x,y) \right) \left( \int_\bbT \Gamma(\theta^{x,V}_s,\theta) \mu^{y,V}_s (\dd \theta)\right) \dd y\right|^2 \dd x \\
        & = \int_I \left|\sum_{kl} \Gamma_{kl} \, e^{ik \theta^{x,V}_s} \, \int_I \left(W(x,y) - V(x,y) \right) \left( \int_\bbT e^{il\theta} \mu^{y,V}_s (\dd \theta)\right) \dd y\right|^2 \dd x.
    \end{split}
    \end{equation}
Multiplying and dividing by $(kl)^{(1+\varepsilon)/2}$  one is left with
\begin{equation}
    \label{intermedio}
    \begin{split}
        \leq &  \int_I \Bigg|\sum_{kl} \left( (kl)^{(1+\varepsilon)/2} \Gamma_{kl} \, e^{ik \theta^{x,V}_s} \right) \\
        &\left((kl)^{- (1+\varepsilon)/2} \int_I \left(W(x,y) - V(x,y) \right) \left( \int_\bbT e^{il\theta} \mu^{y,V}_s (\dd \theta)\right) \dd y\right) \Bigg|^2 \dd x. \\
        \leq &C_\Gamma \sum_{kl} (kl)^{-1-\varepsilon} \int_I \left|\int_I \left(W(x,y) - V(x,y) \right) \left( \int_\bbT e^{il\theta} \mu^{y,V}_s (\dd \theta)\right) \dd y\right|^2 \dd x \\
    \end{split}
    \end{equation}
where in the second step we have applied Cauchy-Schwartz inequality and \eqref{gammaAbsConv}. Using that $W$ and $V$ are bounded, as well as the fact that
    \begin{equation}
      \left|\int_I \left(W(x,y) - V(x,y) \right) \left( \int_\bbT e^{il\theta} \mu^{y,V}_s (\dd \theta)\right) \dd y\right| \leq 1,
    \end{equation}
    we conclude
    \begin{equation}
    \begin{split}
        G &\leq C \sup_{\norm{a}_\infty, \norm{b}_\infty \leq 1} \int_I \left|\int_I \left(W(x,y) - V(x,y) \right) \left( a(y) + i b(y) \right) \dd y\right| \dd x \\
        & \leq C \norm{W-V}_{\infty \to 1}.
    \end{split}
    \end{equation}
    Since the norm $\norm{\,\cdot\,}_{\infty\to 1}$ is equivalent to the cut-norm \eqref{d:cutnormW}, the proof is concluded.
\end{proof}

\subsection{Proofs for the non-linear process \eqref{eq:nonLin2}}
\label{pss:nonLin2}

\medskip

\begin{proof}[Proof of Theorem \ref{thm:nonLin2}]
    The first part follows directly from Proposition \ref{p:nonLin1} and Remark \ref{rem:varphi}. The proof of \eqref{eq:pde_mean} is similar to the proof of \eqref{eq:pde}, but note that we are now integrating with respect to the randomness in $U$ as well.
\end{proof}

\medskip

\begin{proof}[Proof of Theorem \ref{thm:sta2}]
    Let $\theta^{U,W}$ and $\theta^{U,V}$ be the two solutions to \eqref{eq:nonLin2} associated to $W$ and $V$ respectively, coupled by taking the same uniform random variable $U$. Let $\mu^{x,W}$ and $\mu^{x,V}$ represent the laws of $\theta^{U,W}$ and $\theta^{U,V}$ conditioned on $U=x$, for $x \in I$.
    
    Consider $\varphi \in S_I$ an invertible measure preserving map. Recall that $\theta^{\varphi(U),V}$ also satisfies equation \eqref{eq:nonLin2} with $V^\varphi$, see Remark \ref{rem:varphi}. We compare the trajectories $\theta^{U,W}$ and $\theta^{\varphi(U), V}$.
    
    Consider the difference between the equations satisfied by $\theta^{U,W}$ and $\theta^{\varphi(U), V}$, add and subtract the term $ W(U,y) \Gamma(\theta^{\varphi(U),V}_r, \theta) (\mu^{y,W}_r - \mu^{\varphi(y),V}_r)$ to obtain that

    \begin{equation}
    \begin{split}
        &\left|\theta^{U,W}_s - \theta^{\varphi(U),V}_s \right|^2 \leq C \int_0^s \left| F(\theta^{U,W}_r) - F(\theta^{\varphi(U),V}_r)\right|^2 \dd r \\
        & + C \int_0^s \left|\int_I W(U,y) \int_\bbT \Big(\Gamma(\theta^{\varphi(U),W}_r, \theta)- \Gamma(\theta^{\varphi(U),V}_r, \theta)\Big) \mu^{\varphi(y),W}_r (\dd \theta) \, \dd y \right|^2 \dd r  \\
        & + C \int_0^s \left|\int_I W(U,y) \int_\bbT \Gamma(\theta^{\varphi(U),V}_r, \theta) (\mu^{y,W}_r - \mu^{\varphi(y),V}_r) (\dd \theta) \, \dd y \right|^2 \dd r  \\
        & + C \int_0^s \left| \int_I (W(U,y) - V^\varphi(U,y)) \int_\bbT \Gamma(\theta^{\varphi(U),V}_r, \theta) \mu^{\varphi(y),V}_r (\dd \theta) \, \dd y \right|^2 \dd r.
    \end{split}
    \end{equation}
  The first two integrals on the r.h.s. are bounded by $C \int_0^s \left| \theta^{U,W}_r - \theta^{\varphi(U),V}_r\right|^2 \dd r $, using that $F$ and $\Gamma$ are Lipschitz-continuous.
  While the third integral in the r.h.s. can be estimated using \eqref{d:lips_dist} and the fact that $0 \leq W \leq 1$. Thus we get
    \begin{equation}
    \begin{split}
        & \left|\int_I W(U,y) \int_\bbT \Gamma(\theta^{\varphi(U),V}_r, \theta) (\mu^{y,W}_r - \mu^{\varphi(y),V}_r) (\dd \theta) \, \dd y \right|^2 \\
        & \leq \int_I D^2_r(\mu^{y, W}, \mu^{\varphi(y),V}) \, \dd y = d^2_r \left(\mu^W, (\mu^V)^\varphi\right),
    \end{split}
    \end{equation}
    where we have used the notation $(\mu^V)^\varphi$ for $\{\mu^{\varphi(y),V}\}_{y \in I}$.
    
    Taking the supremum over $s\in[0,t]$ and the expectation with respect to the Brownian motions, the initial conditions and the random variable $U$, we obtain
    \begin{equation}
    \label{nl:double_gron}
        \begin{split}
        \int_I \bE \left[ \sup_{s \in [0,t]} \left| \theta^{x,W}_s - \theta^{\varphi(x),V}_s \right|^2 \right] \dd x \leq & C \int_0^t \int_I \bE \left[ \sup_{r \in [0,s]} \left| \theta^{x,W}_r - \theta^{\varphi(x),V}_r \right|^2 \right] \dd x \, \dd s\\
        & + C \int_0^t d^2_s\left(\mu^W, (\mu^V)^\varphi\right) \dd s + C G,
        \end{split}
    \end{equation}
  %
    where $G$ is given by
    \begin{equation}
        G = \int_0^t \bE\left[ \int_I \left| \int_I (W(x,y) - V^\varphi(x,y)) \int_\bbT \Gamma(\theta^{\varphi(x),V}_s, \theta) \mu^{\varphi(y),V}_s (\dd \theta) \dd y \right|^2 \dd x \right]\dd s.
    \end{equation}
    In the proof of Proposition \ref{p:sta1} we proved the following estimates:
    \begin{equation}
    \begin{split}
        d^2_t \left(\mu^W, (\mu^V)^\varphi\right) & \leq \int_I \bE \left[ \sup_{s \in [0,t]} \left| \theta^{x,W}_s - \theta^{\varphi(x),V}_s \right|^2 \right] \dd x, \\
        G\, & \leq \, C' \norm{W- V^\varphi}_\square, \quad \text{ for some } C'>0.
    \end{split}
    \end{equation}
    Applying these bounds to \eqref{nl:double_gron} and using Gronwall's inequality twice as in the previous proof, yields
    \begin{equation}
        d^2_t \left(\mu^W, (\mu^V)^\varphi\right) \leq C \norm{W-V^\varphi}_\square.
    \end{equation}
    By taking the infimum with respect to $\varphi \in S_I$ and recalling the definition of the cut-distance \eqref{d:cut-distance} together with \eqref{eq:dis_inequality}, we obtain
    \begin{equation}
        D_t(\bar{\mu}^W, \bar{\mu}^V) \leq \tilde{d}_t \left(\mu^W, \mu^V\right) \leq C \, \delta_\square (W,V)^{1/2}.
    \end{equation}
  The proof is concluded.
\end{proof}

\section{Proof of Theorem \ref{thm:conv}}
\label{pss:empMes}
In order to prove Theorem \ref{thm:conv}, we couple the system \eqref{eq:g} to a sequence of identically distributed copies of the non-linear process $\theta$, which is obtained by sampling $\{U_i\}_{i \in \N}$ IID uniform random variables and choosing the same initial conditions and Brownian motions of \eqref{eq:g}.

For every $i \in \N$, denote these copies by $\theta^i = \theta (U_i)$. In particular, $\theta^i$ is defined as the solution for $t \in [0,T]$ to
\begin{equation}
\label{eq:coupledNonLin}
    \theta^i_t = \theta^i_0 + \int_0^t F(\theta^i_s) \dd s + \int_0^t \int_I W(U_i,y) \int_\bbT \Gamma(\theta^i_s, \theta) \mu^y_s (\dd \theta) \dd y \, \dd s + B^i_t. \\
\end{equation}
Observe that $\{\theta^i \}_{i \in \N}$ is an exchangeable sequence and, in particular, that the variables $\theta^i$ are independent random variables when conditioned on the randomness of $W$.


Before the proof of Theorem \ref{thm:conv}, we give a trajectorial estimate.
\begin{lem}
    \label{lem:trajEst}
    Under the hypothesis of Theorem \ref{thm:conv}, it holds that
    \begin{equation}
        \lim_{n \to \infty} \bbE \times \bE \left[ \frac 1n \sum_{i=1}^n \sup_{t \in [0,T]} \left| \theta^{i,n}_t - \theta^i_t \right|^2 \right] = 0.
    \end{equation}
\end{lem}

\begin{proof}
As done before, we compare the trajectories $\theta^{i,n}$ and $\theta^i$, by studying the equation satisfied by $\left|\theta^{i,n}_s - \theta^i_s \right|^2$, recall \eqref{eq:g} and \eqref{eq:coupledNonLin}. Add and subtract in the integrals the term $ \left( \xi^{(n)}_{ij} - W(U_i, U_j) \right) \Gamma(\theta^i_r, \theta^j_r)$ so as to get
    \begin{equation}
    \begin{split}
        \left|\theta^{i,n}_s - \theta^i_s \right|^2 & \leq C \int_0^s \left| F(\theta^{i,n}_r) - F(\theta^i_r) \right|^2 \dd r \\
        & + C \int_0^s \left| \frac 1n \sum_{j=1}^n \xi^{(n)}_{ij} \left(\Gamma(\theta^{i,n}_r, \theta^{j,n}_r) - \Gamma(\theta^i_r, \theta^j_r) \right) \right|^2 \dd r \\
        & + C \int_0^s \left| \frac 1n \sum_{j=1}^n \left( \xi^{(n)}_{ij} - W(U_i, U_j) \right) \Gamma(\theta^i_r, \theta^j_r) \right|^2  \dd r \\
        & + C \int_0^s \left| \frac 1n \sum_{j=1}^n W(U_i, U_j) \Gamma(\theta^i_r, \theta^j_r) -  \int_I W(I_i, y) \int_\bbT \Gamma(\theta^i_r, \theta) \mu^y_r (\dd \theta) \, \dd y \right|^2 \dd r.
    \end{split}
    \end{equation}
    We now use the Lipschitz property of $\Gamma$ and $F$, sum over $i$ and take the supremum over $s\in[0,t]$, together with the expectation $\bbE \times \bE$, which we just write $E$ for simplicity,
        \begin{equation}
    \begin{split}
        &E \left[\frac 1n \sum_{i=1}^n \sup_{s \in [0,t]} \left|\theta^{i,n}_s - \theta^i_s \right|^2\right] \leq  C \int_0^t E\left[ \frac 1n \sum_{i=1}^n \sup_{q \in [0,r]} \left| \theta^{i,n}_q - \theta^i_q \right|^2\right] \dd r \\
        &+ C \int_0^t E \left[\frac 1n \sum_{i=1}^n \left| \frac 1n \sum_{j=1}^n \left( \xi^{(n)}_{ij} - W(U_i, U_j) \right) \Gamma(\theta^i_r, \theta^j_r) \right|^2\right]  \dd r \\
        &+ C \int_0^t \frac 1n \sum_{i=1}^n E \left[\left| \frac 1n \sum_{j=1}^n W(U_i, U_j) \Gamma(\theta^i_r, \theta^j_r) -  \int_I W(U_i, y)\int_\bbT \Gamma(\theta^i_r, \theta) \mu^y_r (\dd \theta) \, \dd y \right|^2\right] \dd r.
    \end{split}
    \end{equation}
    Observe that the last term is bounded by a constant divided by $n$ since by taking the conditional expectation with respect to $\theta^j$ and $U^j$, one obtains {\change for $i\neq j$}
    \begin{equation}
        \bE \left[ W(U_i, U_j) \Gamma(\theta^i_s, \theta^j_s) \right] = \int_I W(U_i, y) \int_\bbT \Gamma(\theta^i_s, \theta) \mu^y_s (\dd \theta) \, \dd y
    \end{equation}
    and, conditionally on $W$, the random variables $\{\theta^i\}_{i \in \N}$ are IID.
    
    Turning to the second term, we will prove that
    \begin{equation}
    \label{conv:graph}
        E \left[\frac 1n \sum_{i=1}^n \left| \frac 1n \sum_{j=1}^n \left( \xi^{(n)}_{ij} - W(U_i, U_j) \right) \Gamma(\theta^i_s, \theta^j_s) \right|^2\right]  \leq C \, \bbE \left[ \delta_\square (\xi^{(n)}, W^{(n)})\right] + o(1),
    \end{equation}
    where $W^{(n)}:=\{W(U_i,U_j)\}_{i,j = 1, \dots, n}$ is a $W$-random graph with $n$ vertices, see \eqref{d:WrandomGraph}.
    This, together with a Gronwall argument implies that
    \begin{equation}
        E \left[\frac 1n \sum_{i=1}^n \sup_{s \in [0,t]} \left|\theta^{i,n}_t - \theta^i_t \right|^2\right] \leq C \, \bbE \left[ \delta_\square (\xi^{(n)}, W^{(n)})\right] + o(1)
    \end{equation}
    and the claim follows by taking the limit for $n$ which tends to infinity and the fact that $W^{(n)}$ converges $\bbP$-a.s.~to $W$, recall Theorem \ref{thm:DJ}.
    
    \medskip
    
    Turning to \eqref{conv:graph}, we use an argument similar to \eqref{eq:G}--\eqref{gammaAbsConv}. Recall that since $\Gamma \in \cC^{1+\varepsilon}$, it admits a Fourier series \eqref{gammaFourier} with coefficients $\Gamma_{kl}$ such that
    \[
    \sum_{k,l \in \Z} (kl)^{1+\varepsilon} |\Gamma_{kl}|^2 < \infty.
    \]
    Plugging its Fourier expression in the left-hand side of \eqref{conv:graph}, multiplying and dividing by $(kl)^{(1+\varepsilon)/2}$, we get
    \begin{equation}
    \begin{split}
        &E \left[\frac 1n \sum_{i=1}^n \left| \frac 1n \sum_{j=1}^n \left( \xi^{(n)}_{ij} - W(U_i, U_j) \right) \Gamma(\theta^i_s, \theta^j_s) \right|^2\right] \\
        &= E \left[\frac 1n \sum_{i=1}^n \left|  \frac 1n \sum_{j=1}^n \left( \xi^{(n)}_{ij} - W(U_i, U_j) \right) \sum_{k,l} \Gamma_{kl} e^{i \theta^i_s k} e^{i\theta^j_s l} \right|^2\right] \\
        &\leq C E \left[ \sum_{k,l} (kl)^{-1-\varepsilon} \frac 1n \sum_{i=1}^n \left|  \frac 1n \sum_{j=1}^n \left( \xi^{(n)}_{ij} - W(U_i, U_j) \right)  e^{i \theta^i_s k} e^{i\theta^j_s l} \right|^2\right],
    \end{split}
    \end{equation}
    where we have used Cauchy-Schwartz inequality as in the proof of Theorem \ref{thm:sta2}. Observe that $\sum_{kl} (kl)^{-1-\varepsilon}$ is convergent and that $\left| e^{i \theta^i_s k} \right| \leq 1$ for all $k$ and $s$: we can thus bound $\bP$-a.s. the previous term by
    \begin{equation}
        \bbE \left[ \sup_{s_i,t_j \in \{\pm 1\}} \left|  \frac 1{n^2} \sum_{i,j=1}^n \left( \xi^{(n)}_{ij} - W(U_i, U_j) \right)  s_i t_j \right|\right].
  \end{equation}
Recall that $W^{(n)} = \{W(U_i,U_j)\}_{i,j = 1, \dots, n}$ is a $W$-random graph with $n$ vertices. Since the particles $\{\theta^i\}_{i \in \N}$ are exchangeable, every computation done so far holds no matter the order of $\{\theta^i\}_{i=1, \dots, n}$ and, in particular, of $\{U^i\}_{i =1, \dots, n}$. In particular, the last inequality holds for every relabeling of $W^{(n)}$.
    
    From the definition of $\hat{\delta}_\square$ \eqref{d:cut-distanceGraph}, one can thus take the labeling of $\{U_i\}_{i =1, \dots, n}$ for every $n \in \N$, such that
    \begin{equation}
            \bbE \left[ \sup_{s_i,t_j \in \{\pm 1\}} \left|  \frac 1{n^2} \sum_{i,j=1}^n \left( \xi^{(n)}_{ij} - W(U_i, U_j) \right)  s_i t_j \right|\right] = \bbE \left[ \hat{\delta}_\square (\xi^{(n)}, W^{(n)}) \right].
    \end{equation}
    Using the asymptotic equivalence of $\hat{\delta}_\square$ with $\delta_\square$, see Remark \ref{rem:cut-distance}, the claim is proved and the proof is concluded.
\end{proof}

\medskip

\begin{proof}[Proof of Theorem \ref{thm:conv}]
    The equivalence between the convergence in $\bbP$-probability  of $\xi^{(n)}$ and equation \eqref{eq:graphConv} is proven in Lemma \ref{lem:graphConv}. We turn to the proof of the convergence of $\mu^n$.
    
    It is well known that the bounded Lipschitz distance, recall \eqref{d:lips_dist}, metricizes the weak convergence and defines a distance between probability measures. In particular, in order to show that $\mu^n$ converges in $\bP \times \bbP$-probability to $\bar{\mu}$ in $\cP(\cC([0,T], \bbR))$, it is enough to prove that
    \begin{equation}
    \lim_{n \to \infty} \bbE \times \bE \left[\left| \int f (\theta) \mu^n (\dd \theta)- \int f (\theta) \bar{\mu} (\dd \theta)\right| \right]  =0,
    \end{equation}
    for every $f$ bounded and Lipschitz function with values in $\cC([0,T], \bbR)$.
    
    Using the fact that $\bar{\mu}$ is the law of $\{\theta^i\}_{i \in \N}$ (recall \eqref{eq:coupledNonLin}), it is enough to show that
    \begin{equation}
        \lim_{n \to \infty} \frac 1n \sum_{j=1}^n \bbE \times \bE \left[\left|f(\theta^{j,n}) - f(\theta^j)\right|\right] = 0.
    \end{equation}
    This is implied by the fact that $f$ is Lipschitz and by Jensen's inequality. Indeed,
    \begin{equation}
    \frac 1n \sum_{j=1}^n \bbE \times \bE \left[\left|f(\theta^{j,n}) - f(\theta^j)\right|\right] \leq  \bbE \times \bE \left[\frac 1n \sum_{j=1}^n \sup_{t \in [0,T]} \left| \theta^{j,n}_t - \theta^j_t \right|^2 \right] ^{1/2},
    \end{equation}
    which goes to zero as $n\to\infty$ by Lemma \ref{lem:trajEst}.
\end{proof}

\appendix

\section{Graph convergence and random graphons}
\label{annex:graph}

\subsection{Convergence in probability}
\label{pss:graphs}
The characterization of the convergence in distribution for a sequence of graphs has been originally given in \cite{cf:DJ08}. We give here a useful notion of convergence in $\tilde{\cW}_0$ by means of the cut-distance $\delta_\square$, which is equivalent to the convergence in probability for graph sequences.

\begin{lem}
    \label{lem:graphConv}
    Assume that $\{\xi^{(n)}\}_{n \in \N}$ is a sequence of random graphs and $W$ a random graphon in $\tilde{\cW}_0$. Then, $\xi^{(n)}$ converges in $\bbP$-probability to $W$ if and only if \eqref{d:graphConv} holds, i.e., if and only if
    \begin{equation*}
    \lim_{n \to \infty} \bbE \left[ \delta_\square \left(\xi^{(n)}, W \right)\right] = 0.
    \end{equation*}
\end{lem}

\begin{proof}
    Recall that $(\tilde{\cW}_0, \delta_\square)$ is a compact metric space, so that the convergence of $\xi^{(n)}$ in probability is equivalent to
    \begin{equation}
    \label{d:convProGraph}
    \forall \varepsilon >0, \quad \lim_{n \to \infty} \bbP \left( \delta_{\square} (\xi^{(n)}, W) > \varepsilon \right) = 0.
    \end{equation}

    Observe that the sequence of positive real random variables $\{ \delta_{\square} (\xi^{(n)}, W)\}_{n \in \N}$ is uniformly bounded by 1. Equation \eqref{d:convProGraph} is then equivalent to the convergence in $L^1$, i.e., equivalent to \eqref{d:graphConv}.
\end{proof}

\bigskip

\section*{Acknowledgements}
FC is thankful to his supervisor  Giambattista Giacomin for insightful discussions and advice. FC acknowledges the support from the European Union’s Horizon 2020 research and innovation programme under the Marie Sk\l odowska-Curie grant agreement No 665850. FRN was partially supported by the Netherlands Organisation for Scientific Research (NWO) [Gravitation Grant number 024.002.003--NETWORKS].

\bibliographystyle{abbrv}
\bibliography{biblio}
\end{document}